\documentclass{scrartcl}

\newif\ifarxiv
\arxivtrue

\usepackage[english]{babel}
\usepackage{tkz-berge}
\usepackage{amsmath,amssymb,amsthm}
\usepackage{graphicx}
\usepackage{algorithm}
\usepackage{algpseudocode}
\usepackage[T1]{fontenc}
\usepackage{lmodern}
\usepackage{enumerate}
\usepackage{cite}
\usepackage{mathrsfs}
\usepackage{url}
\usepackage{comment}
\usepackage{mathrsfs}
\usepackage{subfigure}

\newtheorem{lemma}{Lemma}[section]
\newtheorem{theorem}[lemma]{Theorem}

\newtheorem{corollary}[lemma]{Corollary}
\theoremstyle{remark}
\newtheorem{remark}[lemma]{Remark}

\newcommand*{\DRel}{\operatorname{DRel}}

\title{Domination Reliability}
\author{Klaus Dohmen and Peter Tittmann\\ Hochschule Mittweida\\ Technikumplatz 17\\ 09648 Mittweida, Germany}

\begin{document}

\maketitle

\begin{abstract}
\noindent 
In this paper, we propose a new network reliability measure for some particular kind of service networks, 
which we refer to as \emph{domination reliability}.
We relate this new reliability measure to the domination polynomial of a graph and the coverage probability of a hypergraph. 
We derive explicit and recursive formul\ae{} for domination reliability and its associated domination reliability polynomial,
deduce an analogue of Whitney's broken circuit theorem,
and prove that computing domination reliability is NP-hard.
\end{abstract}

\section{Introduction}
\label{intro}

All graphs considered in this paper are finite, undirected, and simple. 
We write $G=(V,E)$ to denote that $G$ is a graph having vertex set $V$ and edge set $E$.
We also use $V(G)$ and $E(G)$ to denote the vertex set and edge set of $G$, respectively.
Throughout, we assume that the vertices of $G$ are subject to random and independent failure 
according to some given probability distribution, 
whereas the edges are perfectly reliable. 
One can think of each vertex as a service provider offering some kind of service to that vertex and its neighbours.
Failure of a vertex does not mean that the vertex is not alive; 
it just means that the service provider on that vertex is not available. 
If that happens, the service provider of a neighbouring vertex will do the job provided, 
of course, it is available. 
The probability that each vertex is served by an available service provider 
thus corresponds to the probability that the operating vertices of the graph constitute 
a \emph{dominating set} of the graph, that is, 
a subset $X$ of $V$ where each vertex $v\in V\setminus X$ is adjacent to some vertex in $X$,
in which case we say that $v$ is \emph{dominated} by $X$.
We refer to this probability as the \emph{domination reliability} of $G$. 
Note that this quantity depends on $G$ and the individual vertex operation probabilities $p_{v}=1-q_{v}$, $v\in V$.
We use $\DRel(G,\mathbf{p})$ where $\mathbf{p}=(p_{v})_{v\in V}$ to denote the domination reliability of $G$. 
If all vertex operation probabilities are equal to $p$, 
we write $\DRel(G,p)$ rather than $\DRel(G,\mathbf{p})$.
Throughout this paper, $\mathbf{q} = \mathbf{1}-\mathbf{p}$ for any $\mathbf{p}\in [0,1]^V$,
and $q=1-p$ for any $p\in [0,1]$.

The concept of domination reliability is closely related to that of the
domination polynomial introduced by \textsc{Arocha} and \textsc{Llano} \cite{ArLl00}; 
see also \cite{AAP09,AO09,AP09a,AP09b} for some recent results.

This paper is organized as follows.
In Sections \ref{Decomposition} and \ref{InclusionExclusion} 
we consider the general case where vertex operation probabilities need not be equal.
In Section \ref{Decomposition}
we derive splitting formul\ae{} for sums and joins of vertex-disjoint graphs and a vertex decomposition formula.
The latter forms the basis of a recursive algorithm for computing a modified domination reliability measure for bipartite graphs, 
which in turn can be used to compute the domination reliability of any graph by constructing its neighbourhood graph.
In Section \ref{InclusionExclusion} we derive an inclusion-exclusion expansion for the general case,
and an analogue of Whitney's broken circuit theorem, which is well-known in chromatic graph theory.
In Section \ref{ReliabilityPolynomial} we consider the particular case where all vertex operation probabilities are equal to $p$,
where $p$ is some formal constant. 
It turns out that, in this case, the domination reliability of any graph is a polynomial in $p$,
which we refer to as the \emph{domination reliability polynomial}.
We give explicit and recursive formul\ae{} for the domination reliability polynomial for particular classes of graphs, 
and show that some non-isomorphic graphs cannot be distinguished by their domination reliability polynomial.
In Section \ref{DominationPolynomial} we establish a close relationship with the aforementioned domination polynomial 
due to \textsc{Arocha} and \textsc{Llano} \cite{ArLl00},
and draw some conclusions from our results on domination reliability to the domination polynomial.
In Section \ref{Complexity} we prove that computing the domination reliability polynomial, 
and hence domination reliability in general, is NP-hard.
Section \ref{Coverage} is devoted to the hypergraph reliability covering problem introduced by 
\textsc{Ball}, \textsc{Provan}, and \textsc{Shier} \cite{BPS91} (see also \cite{Shier91}),
which is shown to be equivalent to computing domination reliability of graphs.
The final section contains some concluding remarks.

\section{Decomposition techniques}
\label{Decomposition}

We start our investigation of domination reliability 
with the general case where the vertex operation probabilities need not be equal.
Recall from the above that we assume that the associated events are mutually independent. 

Subsequently, we derive decomposition formul\ae{} for sums and joins of vertex-disjoint graphs.
Recall that the \emph{sum} of any two vertex-disjoint graphs $G=(V,E)$ and $H=(W,F)$ is defined by $G + H:=(V\cup W,E\cup F)$.
The \emph{join} $G\ast H$ of any two vertex-disjoint graphs $G$ and $H$ is obtained from their sum 
by adding an edge between any vertex of $G$ and any vertex of $H$. 
In the particular case, where $G$ and $H$ are edgeless graphs on $s$ and $t$ vertices, respectively,
$G\ast H$ is isomorphic to the complete bipartite graph $K_{s,t}$.

\begin{theorem}
\label{th_union_join}
For any vertex-disjoint graphs $G$ and $H$ and any $\mathbf{p}\in [0,1]^{V(G + H)}$,
\begin{equation}
 \DRel(G + H,\mathbf{p}) = \DRel(G,\mathbf{p}) \DRel(H,\mathbf{p}) \, ,
\end{equation}
and
\begin{multline}
\label{drelast}
\DRel(G\ast H,\mathbf{p}) =  \left[1-\prod\limits_{v\in V(G)} q_{v}\right] \left[1 - \prod\limits_{w\in V(H)} q_{w}\right] \\
+ \left( \prod\limits_{w\in V(H)} q_{w} \right) \DRel(G,\mathbf{p}) + \left( \prod\limits_{v\in V(G)} q_{v} \right) \DRel(H,\mathbf{p})
\end{multline}
where $q_v = 1-p_v$ for any $v\in V(G\ast H)$.
\end{theorem}

\begin{proof}
The first statement follows immediately from the definition of domination reliability and the $+$ operation.
For the second statement, 
we observe that all vertices of $G\ast H$ are dominated if and only if one of the events $A$, $B$ or $C$ occurs:
\begin{enumerate}[\quad ($A$)]
 \item Neither all vertices of $G$ fail, nor all vertices of $H$ fail.
 \item All vertices of $H$ fail, while all vertices of $G$ are operating or dominated.
 \item All vertices of $G$ fail, while all vertices of $H$ are operating or dominated.
\end{enumerate}
Note that these three events are pairwise disjoint, and that their respective probabilities are given (in the same order) by the right-hand side of (\ref{drelast}).
\end{proof}

Theorem \ref{th_union_join} is most significant for cographs. 
Cographs are a comprehensive class of graphs including all complete graphs,
complete bipartite graphs, threshold graphs, and Tur\'{a}n graphs.
Among the many characterizations of cographs (see e.g., \cite{BBS})
the following best serves our purpose: 
A graph is a cograph if and only if it can be obtained from isolated vertices by a
finite sequence of sum and join operations of vertex-disjoint graphs.
With this characterization, the next corollary easily follows.

\begin{corollary}
\label{polytime}
The domination reliability of cographs can be computed in polynomial time.
\end{corollary}

\begin{remark}
An alternative characterization of cographs is that they are graphs of clique width 2.
In view of this and the general results on clique and tree decompositions in \cite{CMR01},
the statement of Corollary \ref{polytime} can be generalized to graphs of bounded clique width or bounded tree width
if a clique or tree decomposition of minimal width is part of the input.
For the interested reader,
we just mention that the property of being a dominating set can be described in monadic second-order logic (MSOL) by the expression
\[ \textit{dominating}(X) \, \Longleftrightarrow \,\forall v\left(v\in X\vee\exists w\left( w\in X \wedge \textit{adjacent}(v,w)\right) \right) . \]
Since this expression is actually an $\text{MS}_1$-expression (that is, it contains no quantification over edge subsets), 
the general results on $\text{MS}_1$-expressions in \cite{CMR01} apply.
\end{remark}

We proceed by establishing a generally applicable recursive method for computing $\DRel(G,\mathbf{p})$.
In order to formulate the method, we use $N_G[J]$ to denote the \emph{closed neighbourhood} of $J$ in $G$, 
that is, the set of all vertices of $G$ which are in $J$ or adjacent to some vertex in $J$,
and write $N_{G}[v]$ instead of $N_{G}[\{v\}]$ for any single vertex $v$.

Our method requires a generalization of our notion of domination reliability.
For any graph $G=(V,E)$, $X,Y\subseteq V$, and $\mathbf{p}\in [0,1]^V$ we use
$\DRel(G,X,Y,\mathbf{p})$ to denote the probability that 
under the assumption that the vertices of $Y$ fail randomly and independently with probability $p_y$, $y\in Y$,
each vertex $x\in X$ is adjacent to some operating vertex $y\in Y$.
According to this definition, $\DRel(G,\mathbf{p}) = \DRel(G,V,V,\mathbf{p})$.

\begin{theorem}
\label{theo_recurrence}
For any graph $G=(V,E)$, $X,Y\subseteq V$, $y\in Y$, and $\mathbf{p}\in [0,1]^V$, 
\begin{multline*} 
\DRel(G,X,Y,\mathbf{p}) \,=\, p_y \DRel(G, X\setminus N_G[y], Y\setminus \{y\},\mathbf{p}) \\ 
+\,(1-p_y)\DRel(G,X,Y\setminus\{y\},\mathbf{p})\, ,
\end{multline*}
provided $X\subseteq N_G[Y]$ and $X\neq\emptyset$. 
Otherwise, $\DRel(G,X,Y,\mathbf{p})$ equals $0$ resp\@. $1$.
\end{theorem}

\begin{proof}
The main case follows from the law of total probability:
If $y$ is operating, which happens with probability $p_y$,
then it dominates all vertices in $N_G[y]$,
so the vertices in $X\setminus N_G[y]$ need to be dominated by $Y\setminus \{y\}$,
which happens with probability $\DRel(G,X\setminus N_G[y], Y\setminus \{y\},\mathbf{p})$.
If $y$ is not operating, which happens with probability $1-p_y$,
then the vertices in $X$ need to be dominated by $Y\setminus\{y\}$,
which happens with probability $\DRel(G,X,Y\setminus\{y\},\mathbf{p})$.
The particular cases $X\not\subseteq N_G[Y]$ and $X=\emptyset$ are obvious.
\end{proof}

\begin{remark}
Theorem \ref{theo_recurrence} gives rise to a recursive algorithm for computing $\DRel$.
For efficiency,
it is recommended to choose $y\in Y$ of maximum degree in each recursive invocation,
and to encode the parameters $X$ and $Y$ as bit vectors of length $|V|$.
\end{remark}

We close this section with discussing a related reliability measure for bipartite graphs.

\medskip

Let $G=(V,W,E)$ be a bipartite graph with vertex-set $V\cup W$ and edge-set $E\subseteq \{\{v,w\}\mathrel| v\in V,\, w\in W\}$.
The vertices in $W$ are assumed to operate randomly and independently with known probabilities $p_w$, $w\in W$.
For $\mathbf{p} = (p_w)_{w\in W}\in [0,1]^{W}$ we denote by
$\DRel^{\leftarrow}(G,\mathbf{p})$ the probability that each vertex in $V$ is adjacent to some operating vertex in $W$.
Accordingly, $\DRel^{\rightarrow}(G,\mathbf{p})$ can be defined by exchanging the roles of $V$ and $W$.
As before, we write $q_v = 1-p_v$ for any $v\in V$ resp\@. $W$, and $\mathbf{q} = \mathbf{1}-\mathbf{p}$.
In order to keep the notation simple,
we allow $\mathbf{p}$ to contain redundant information.

The following theorem provides an analogue of Theorem \ref{theo_recurrence} for $\DRel^{\leftarrow}(G,\mathbf{p})$.
Note that by exchanging the roles of $V$ and $W$, an analogue for $\DRel^{\rightarrow}(G,\mathbf{p})$ is obtained.

\begin{theorem}
\label{drelbipartite}
For any bipartite graph $G=(V,W,E)$, any $w\in W$, and any $\mathbf{p}\in [0,1]^W$,
\[ \DRel^{\leftarrow}(G,\mathbf{p}) \,=\, p_w \DRel^{\leftarrow}(G-w-N_G[w],\mathbf{p}) + (1-p_w)\DRel^{\leftarrow}(G-w,\mathbf{p})\, , \]
provided $V\subseteq N_G[W]$ and $V\neq \emptyset$.
Otherwise, $\DRel^{\leftarrow}(G,\mathbf{p})$ equals $0$ resp\@. $1$.
\end{theorem}

\begin{proof}
Analogous to the proof of Theorem~\ref{theo_recurrence} with $w$ instead of $y$, and $V$ and $W$ in place of $X$ and $Y$, respectively.
\end{proof}
\enlargethispage*{2\baselineskip}

Theorem~\ref{drelbipartite} gives rise to the following algorithm for computing $\DRel^{\leftarrow}(G,\mathbf{p})$.

\begin{algorithm}[!ht]
\caption{Right to left domination reliability of a bipartite graph}
\label{alg1}
\begin{algorithmic}[1]
\Function{$\DRel^{\leftarrow}$}{$G,\mathbf{p}$}{}
\Comment{$G=(V,W,E)$ \hskip1pt a bipartite graph} \\
\textbf{begin}
\Comment{$\mathbf{p}=(p_{w})_{w\in W}$ vertex reliabilities} 
\If{$V = \emptyset$} 
\State \Return{1}
\EndIf
\If{$\min\{\deg(v) \mid v \in V \}=0$} 
\State \Return{0}
\EndIf
\State Choose an arbitrary vertex $w\in W$ of maximum degree
\State \Return{$p_w \DRel^{\leftarrow}(G-w-N_G[w],\mathbf{p}) + (1-p_w)\DRel^{\leftarrow}(G-w,\mathbf{p})$}
\EndFunction
\end{algorithmic}
\end{algorithm}

\begin{remark}
\label{neighbourhoodmethod}
Algorithm \ref{alg1} can be used to compute $\DRel(G,\mathbf{p})$ for an arbitrary graph $G=(V,E)$
and $\mathbf{p}=(p_v)_{v\in V}\in [0,1]^V$
by applying it to the \emph{neighbourhood graph} of $G$.
By definition, this is the bipartite graph $N(G)=(V',W',E')$ where
$V'=\{(v,0)\mathrel| v\in V\}$,
$W'=\{(v,1)\mathrel| v\in V\}$, and
$E'=\{\{(v,0),(w,1)\}\mathrel| \text{$v=w$ or $\{v,w\}\in E$}\}$.
If we put $p_{(v,1)} := p_v$ for any $v\in V$, and $\mathbf{p}' = (p_w)_{w\in W'}$,
then the output $\DRel^{\leftarrow}(N(G),\mathbf{p}')$ of Algorithm \ref{alg1} coincides with 
the domination reliability $\DRel(G,\mathbf{p})$ of $G$.
\par
The time complexity of this method is, unfortunately, exponential in the size of the graph.
The results in Section \ref{Complexity} show that under reasonable assumptions
there is unlikely to exist a polynomial time algorithm for computing domination reliability.
Nevertheless, the method is applicable for graphs of moderate size, and easy to implement.
\end{remark}

\section{Applying the inclusion-exclusion principle}
\label{InclusionExclusion}

The first part of the following theorem is a straightforward application of the inclusion-exclusion principle.
The second part may be regarded as an analogue of Whitney's broken circuit theorem 
on the chromatic polynomial of a graph \cite{Whitney:1932}.

Let $G=(V,E)$ be a graph whose vertex-set is endowed with a linear ordering relation,
and let $v\in V$.
Motivated by the notion of a broken circuit in \cite{Whitney:1932},
we refer to $N_G[v]\setminus \{v\}$ as a \emph{broken neighbourhood} of $G$ if $v=\max N_G[v]$. 

\begin{theorem}
\label{thm2} 
For any graph $G=(V,E)$ whose vertices fail randomly and
independently with probability $q_{v}=1-p_{v}$ for any $v\in V$,
\begin{equation}
\label{inclusion-exclusion}
\DRel(G,\mathbf{p}) \,=\, \sum_{J\subseteq V}\,(-1)^{|J|}\!\prod_{v\in N_{G}[J]} q_{v} \, ,
\end{equation}
where $\mathbf{p}=(p_{v})_{v\in V}\in [0,1]^V$. 
Moreover, 
if $G$ does not contain isolated vertices,
then for any linear ordering relation on $V$
and any set $\mathscr{X}$ of broken neighbourhoods of~$G$,
\begin{equation}
\label{inclusion-exclusion-improved}
\DRel(G,\mathbf{p}) \, = \! \sum_{J\subseteq V\atop J\not\supseteq X \forall X\in\mathscr{X}} \! (-1)^{|J|}\!\prod_{v\in N_{G}[J]} q_{v} \, .
\end{equation}
\end{theorem}

\begin{proof}
For any vertex $v\in V$ let $A_{v}$ be the event that neither $v$ nor any of its neighbouring vertices is operating; 
formally,
\begin{equation*}
A_{v}:=\{I\subseteq V\,|\,v\notin N[I]\} .
\end{equation*}
Then, $\bigcap_{v\in V}\overline{A_{v}}$ occurs if and only if the operating vertices constitute a dominating set of $G$, 
whereas $\bigcap_{j\in J}A_{j}$ occurs if and only if all vertices in $N_{G}[J]$ fail. 
By the independence assumption, this happens with probability $\prod_{v\in N_{G}[J]}q_{v}$.
The first part of the theorem now follows by applying the inclusion-exclusion principle to the events $A_v$, $v\in V$.
\end{proof}

For the proof of the second part of Theorem \ref{thm2} the following lemma is needed.
For a different application of this lemma in the context of network reliability, see \cite{Dohmen:1999,Dohmen:1998}. 

\begin{lemma}[\hspace{-.005em}\cite{Dohmen:1999}]
\label{hilfslemma}
Let $\{A_v\}_{v\in V}$ be a finite family of events in some probability space $(\Omega,\mathscr{A},\Pr)$.
Furthermore, let $V$ be endowed with a  linear ordering relation, 
and let $\mathscr{X}$ be a set of non-empty subsets of $V$ such that for any $X\in\mathscr{X}$,
\begin{equation}
\label{condition}
\bigcap_{x\in X} A_x \subseteq \bigcup_{v>\max X} A_v \, .
\end{equation}
Then,
\begin{equation}
\label{inclusion-exclusion-lemma}
\Pr\left(\, \bigcap_{v\in V} \complement A_v \right) \,= \!\sum_{J\subseteq V\atop J\not\supseteq X \forall X\in\mathscr{X}} \! (-1)^{|J|}
\Pr\left(\, \bigcap_{j\in J} A_j \right).
\end{equation}
\end{lemma}

\begin{proof}[Proof of Theorem \ref{thm2} (Cont'd).]
Since $G$ does not contain isolated vertices,
the broken neighbourhoods of $G$ are non-empty, and $N_G[v] \subseteq  N_G[ N_G[v]\setminus\{v\}]$ for any $v\in V$.
Let $I\in \bigcap_{x\in X} A_x$ for some $X\in\mathscr{X}$.
Then, $N_G[X]\subseteq V\setminus I$, 
and since $X\in\mathscr{X}$,
$X = N_G[v]\setminus\{v\}$ for some $v\in V$ satisfying $v=\max N_G[v] > \max X$.
Therefore,
$N_G[v] \subseteq  N_G[ N_G[v]\setminus\{v\} ] = N_G[X] \subseteq V\setminus I$
and hence, $I\in A_v$ where $v>\max X$.
Thus, (\ref{condition}) is shown, and hence by Lemma \ref{hilfslemma}, the second part of Theorem \ref{thm2} follows.
\end{proof}

\begin{remark}
\label{remainsvalid1}
The second part of Theorem \ref{thm2} coincides with the first part if $\mathscr{X}$ is empty.
We further remark that by applying Lemma \ref{hilfslemma},
the second part of Theorem \ref{thm2} can be generalized
to any system $\mathscr{X}$ of non-empty subsets of $V=V(G)$ satisfying $\max X < \max c(X)$ for any $X\in\mathscr{X}$
where $c(Y) := \{v\in V \mathrel| N_G[v] \subseteq N_G[Y]\}$ for any $Y\subseteq V$.
This latter condition holds for any broken neighbourhood $X$ of $G$, as well as any subset $X$ of $V$ satisfying $N_G[X]=V$ and $\max V\notin X$.
According to \textsc{Pfaltz} and \textsc{Jamison} \cite{Jamison}, $c$ is a closure operator on $V$,
which is related to digital image processing.
\end{remark}

In the following, we use $\delta(G)$ to denote the minimum degree of $G$.

\begin{corollary}
\label{cor3}
Under the requirements of Theorem~\ref{thm2}, 
if $|V|\ge 1$ and $\delta(G)\ge 1$, then
\[ \DRel(G,\mathbf{p}) \,= \!\! \sum_{J\subseteq V\atop |J|\le |V|-\delta(G)}\!\! (-1)^{|J|} \! \prod_{v\in N_{G}[J]} \! q_{v} 
\,+\, (-1)^{|V|-\delta(G)+1} {|V|-1 \choose \delta(G)-1} \prod_{v\in V} q_v \, .\]
\end{corollary}

\begin{proof}
Evidently, $N_G[J]=V$ for any $J\subseteq V$, $|J|>|V|-\delta(G)$. Therefore,
\[ \DRel(G,\mathbf{p}) \,= \!\! \sum_{J\subseteq V\atop |J|\le |V|-\delta(G)}\!\! (-1)^{|J|} \! \prod_{v\in N_{G}[J]} \! q_{v} 
\,+ \! \sum_{J\subseteq V\atop |J|> |V|-\delta(G)}\!\! (-1)^{|J|} \prod_{v\in V} q_v \, . \]
The result now follows by applying the well-known combinatorial identity
\[ \sum_{k=m}^n (-1)^k {n\choose k} = (-1)^m {n-1\choose n-m} \quad (1\le m\le n) . \qedhere\]
\end{proof}

The following corollary is particularly useful if $G$ is a tree.

\begin{corollary}
\label{remainvalid2}
Under the requirements of Theorem \ref{thm2}, 
if $G$ neither contains isolated vertices nor isolated edges, then
\begin{equation}
\label{corollary-identity}
\DRel(G,\mathbf{p}) \,=\, \sum_{J\subseteq V\setminus A}\,(-1)^{|J|}\!\prod_{v\in N_{G}[J]} q_{v}
\end{equation}
where $A$ is any set of vertices of $G$ which are adjacent to a vertex of degree~1.
\end{corollary}

\begin{proof}
Let $V$ be linearly ordered
such that the vertices of degree 1 form an upset.
Then, any edge $\{v,a\}$ where $v$ is of degree~$1$ and $a\in A$
gives rise to a broken neighbourhood $\{a\}$.
Therefore, with $\mathscr{X} := \{ \{a\}\mathrel| a\in A\}$ 
the corollary follows from Theorem \ref{thm2}.
\end{proof}

\begin{remark}
By the traditional Bonferroni inequalities,
truncating the sum in (\ref{inclusion-exclusion}) to $|J|\le r$ for some non-negative integer $r$
gives upper resp\@. lower bounds to $\DRel(G,\mathbf{p})$, 
depending on whether $r$ is even or odd.
In view of the improved Bonferroni inequalities in \cite{Dohmen:1999},
the same applies---in a more general fashion---to (\ref{inclusion-exclusion-lemma}), and 
thus to (\ref{inclusion-exclusion-improved}) and (\ref{corollary-identity}).
\end{remark}

\section{Domination reliability polynomial}
\label{ReliabilityPolynomial}

In the following, we consider $\DRel(G,p)$ for a graph $G$ and a common vertex operation probability $p$.
It follows from Theorem \ref{thm2} (or Theorem \ref{theo_recurrence}) that $\DRel(G,p)$ is a polynomial in $p$,
which we refer to as the \emph{domination reliability polynomial} of $G$.

In general, the domination reliability polynomial of any graph can be computed using the techniques in Sections \ref{Decomposition} and  \ref{InclusionExclusion}.
As an example, 
the authors applied the method described in Remark \ref{neighbourhoodmethod} to the $5\times5$ grid graph depicted in Figure \ref{gridgraph}.
The result
\begin{align*}
& 22p^{7}+1149p^{8}-305p^{9}-29032p^{10}+115946p^{11}-201109p^{12}+132628p^{13} \\
& +136084p^{14}-414834p^{15}+475677p^{16}-316811p^{17}+117544p^{18}-8108p^{19} \\
& -15506p^{20}+8517p^{21}-2066p^{22}+196p^{23}+12p^{24}-3p^{25}.
\end{align*}
is a polynomial function in $p$, which is shown in Figure \ref{graphplot}. 

\ifarxiv

\begin{figure}[!ht]
\centering
\subfigure[\label{gridgraph} A $5\times 5$ grid graph.]{\includegraphics[bb = 137 613 272 750, clip]{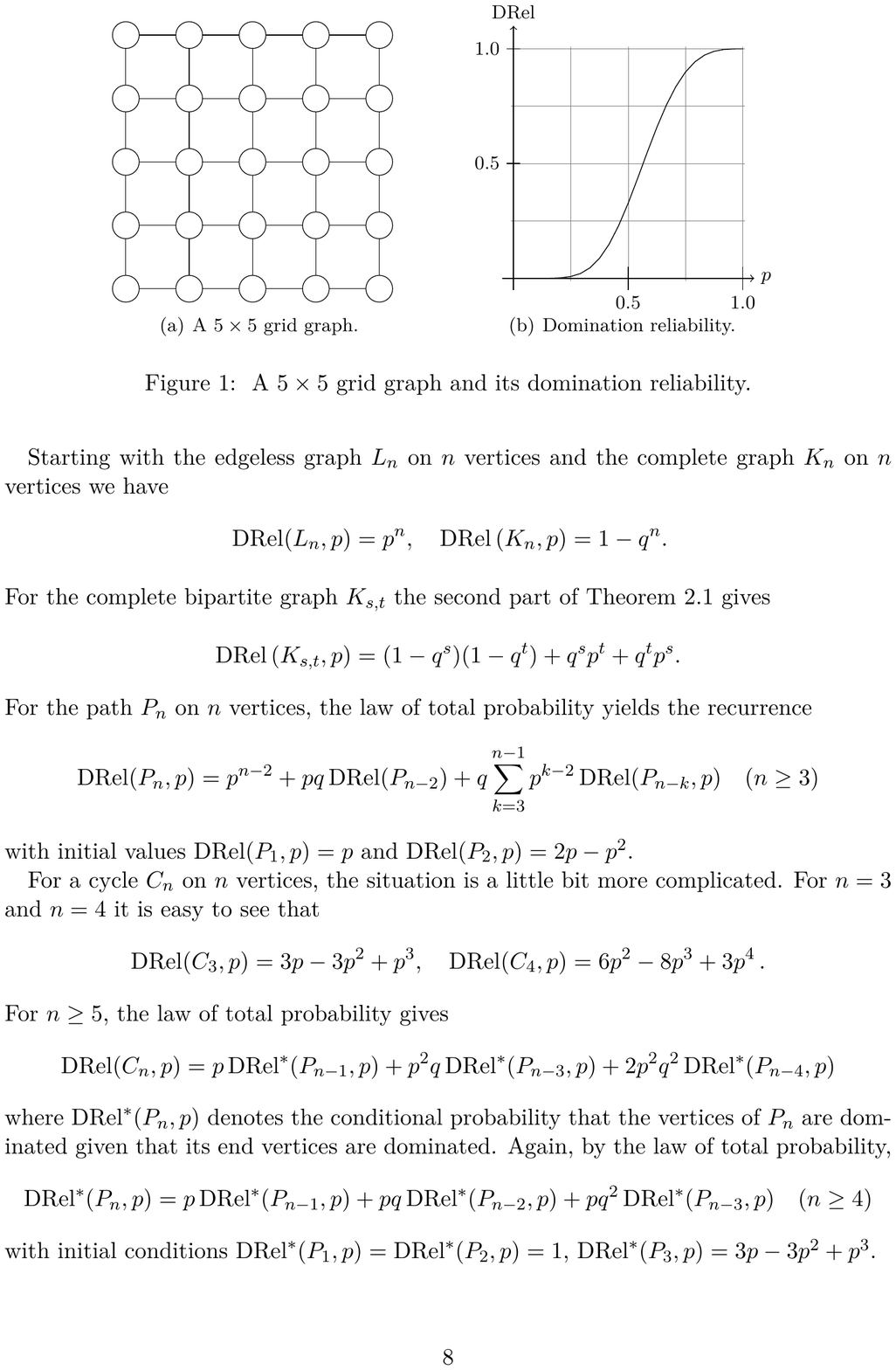}}
\qquad
\subfigure[\label{graphplot} Domination reliability.]{\includegraphics[bb = 306 611 449 756, clip]{Figures.pdf}}
\caption{\label{ReliabilityExample} A $5\times 5$ grid graph and its domination reliability.}
\end{figure}

\else

\begin{figure}[!ht]
\centering
\subfigure[\label{gridgraph} A $5\times 5$ grid graph.]{%
\scalebox{0.7}{%
\begin{tikzpicture}[node distance = .5 cm]
\SetUpEdge[lw = 0.5pt]
\SetVertexNormal[LineWidth=.5pt]
\SetVertexNoLabel
\grEmptyGrid[RA=1.5,RB=1.5]{5}{5}
\foreach \x in {0,1,2,3,4} {
  \foreach \y/\z in {0/1,1/2,2/3,3/4} {
    \Edges(a\x;\y,a\x;\z)
    \Edges(a\y;\x,a\z;\x)
  }
}
\end{tikzpicture}
}
}
\qquad
\subfigure[\label{graphplot} Domination reliability.]{
\footnotesize
\begin{tikzpicture}[domain=0:1,yscale=3.8,xscale=3.8,baseline=-4mm]
\draw[very thin, color=gray] (-0.01,-0.01) grid[step=0.25] (1.01, 1.01);
\draw[->] (-0.05,0) -- (1.05,0) node[right] {$p$};
\draw[->] (0,-0.05) -- (0,1.1) node[above] {$\DRel$};
\draw (0.5,0.05) -- (0.5,-0.05) node[below] {$0.5$};
\draw (1.0,0.05) -- (1.0,-0.05) node[below] {$1.0$};
\foreach \y in {0.5,1.0} {
  \draw (0.03,\y) -- (-0.03,\y) node[left] {$\y$};
}
\draw[color=black] plot function{22*x**7+1149*x**8-305*x**9-29032*x**10+115946*x**11-201109*x**12+132628*x**13+136084*x**14-414834*x**15+475677*x**16-316811*x**17+117544*x**18-8108*x**19-15506*x**20+8517*x**21-2066*x**22+196*x**23+12*x**24-3*x**25};
\end{tikzpicture}
}
\caption{\label{ReliabilityExample} A $5\times 5$ grid graph and its domination reliability.}
\end{figure}

\fi

We proceed our investigation of the domination reliability polynomial by considering some particular classes of graphs,
for which explicit or recursive formul\ae{} can be given.

Starting with the edgeless graph $L_n$ on $n$ vertices and the complete graph $K_n$ on $n$ vertices we have
\begin{align*}
\DRel(L_n,p) & = p^{n}, \quad \DRel\left(K_{n},p\right) = 1-q^{n}. \\
\intertext{For the complete bipartite graph $K_{s,t}$ the second part of Theorem~\ref{th_union_join} gives}
\DRel\left(K_{s,t},p\right) & = (1-q^{s}) (1-q^{t}) + q^{s}p^{t} + q^{t}p^{s}.
\end{align*}
For the path $P_n$ on $n$ vertices, the law of total probability yields the recurrence
\[ \DRel(P_{n},p) = p^{n-2} + pq\DRel(P_{n-2}) + q\sum_{k=3}^{n-1}p^{k-2} \DRel(P_{n-k},p) \quad (n\ge 3) \]
with initial values $\DRel(P_{1},p)=p$ and $\DRel(P_{2},p)=2p-p^{2}$. \par
For a cycle $C_n$ on $n$ vertices, the situation is a little bit more complicated. 
For $n=3$ and $n=4$ it is easy to see that
\[ \DRel(C_{3},p) = 3p - 3p^{2} + p^{3}, \quad \DRel(C_{4},p) = 6p^{2} - 8p^{3} + 3p^{4} \, .  \]
For $n\ge 5$, the law of total probability gives
\[ \DRel(C_{n},p) = p\DRel^\ast(P_{n-1},p) + p^{2}q\DRel^\ast(P_{n-3},p) + 2p^{2} q^{2} \DRel^\ast(P_{n-4},p) \]
where $\DRel^\ast(P_n,p)$ denotes the conditional probability that the vertices of $P_n$ are dominated
given that its end vertices are dominated. 
Again, by the law of total probability,
\[ \DRel^\ast(P_{n},p) = p\DRel^\ast(P_{n-1},p) + pq\DRel^\ast(P_{n-2},p) + pq^{2}\DRel^\ast(P_{n-3},p) \quad (n\ge 4) \]
with initial conditions
$\DRel^\ast(P_{1},p) = \DRel^\ast(P_{2},p)=1$, $\DRel^\ast(P_{3},p) = 3p - 3p^{2} + p^3$.

\begin{remark}
An alternative way to compute the domination reliability of a path on $n$ vertices, which are labelled $1$ through $n$,
is to apply Shier's recursive algorithm \cite[pp.~75--80]{Shier91} (reformulated in \cite[Corollary 2.4]{Dohmen:1998}) to the linearly ordered mincuts 
\[ \{1,2\} \le \{2,3,4\} \le \{3,4,5\} \le \dots \le \{n-3,n-2,n-1\} \le \{n-1,n\} . \]
\end{remark}

\begin{remark}
The domination reliability of a path on $n$ vertices is bounded above by the reliability of a linear consecutive 3-out-of-$n$ failure system.
For a cycle on $n$ vertices, it is equal to the reliability of a circular consecutive 3-out-of-$n$ failure system.
Efficient algorithms for computing these classical reliability measures can be found in \cite{Hwang:1982,Wu-Chen:1993}.
\end{remark}

\begin{remark}
For trees, a similar recurrence as for paths depending on the number of vertices cannot exist.
This is because non-isomorphic trees with equally many vertices may have different domination reliability polynomials.
By exhaustive search, the authors found that this is the case for any pair of non-isomorphic trees with up to six vertices.
This raises the question whether non-isomorphic trees can be distinguished by their domination reliability polynomials. 
This is not the case:
The trees in Figure~\ref{trees} are non-isomorphic, but have the same domination reliability polynomial $4p^{3}-7p^{5}+5p^{6}-p^{7}$.
\end{remark}

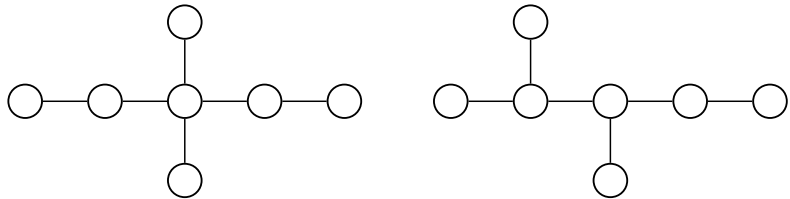
\begin{figure}[ht]
\centering
\scalebox{0.7}{%
\begin{tikzpicture}[node distance = 1.5cm]
\GraphInit[vstyle=Normal]
\SetVertexNoLabel
\Vertices[dir=\EA]{A,B,C,D,E}
\NO(C){F} 
\SO(C){G}
\Edges(A,B,C,D,E)
\Edges(C,F) 
\Edges(C,G)
\tikzset{node distance = 2cm}
\EA(E){H}
\tikzset{node distance = 1.5cm}
\EA(H){I}
\EA(I){J}
\EA(J){K}
\EA(K){L}
\NO(I){M}
\SO(J){N}
\Edges(H,I,J,K,L)
\Edges(I,M)
\Edges(J,N)
\end{tikzpicture}
}
\caption{\label{trees} Non-isomorphic trees with the same domination reliability polynomial.}
\end{figure}

\section{Domination polynomial}
\label{DominationPolynomial}

In this section, we draw some conclusions from our results on domination reliability 
to the domination polynomial of a graph.
This latter polynomial was introduced by \textsc{Arocha} and \textsc{Llano} \cite{ArLl00} for any graph $G=(V,E)$ as the generating function 
\[ D_{G}(x):=\sum_{k=0}^{|V|}d_{k}(G)x^{k} \]
where $d_{k}(G)$ counts the dominating sets of cardinality $k$ in $G$ ($k=0,1,2,\dots$). 

Our first theorem links the domination reliability polynomial of a graph to its domination polynomial.
Thus, we provide a combinatorial interpretation of the latter.

\begin{theorem}
\label{thm1} 
Let $G=(V,E)$ be a graph whose vertices fail randomly and independently with equal probability $q=1-p$. 
Then,
\[ \DRel(G,p) \, = \, q^{n}D_{G}(p/q) \, ,\]
where $n$ denotes the number of vertices in $G$.
\end{theorem}

\begin{proof}
By complete state enumeration,
\begin{equation*}
\DRel(G,p) \, = \!\sum_{\substack{J\subseteq V \\N_{G}[J]=V}}p^{|J|} q^{n-|J|}
\,=\, q^{n}\sum_{k=0}^{n}\sum_{\substack{J\subseteq V \\N_{G}[J]=V \\|J|=k}} \left(\frac{p}{q}\right)^{k}
\,=\, q^{n}\sum_{k=0}^{n}d_{k}(G)\left(\frac{p}{q}\right)^{k},
\end{equation*}
which proves the result.
\end{proof}

\begin{corollary}
\label{zweidarstellungen} For any graph $G$ on $n$ vertices, and any
$p,q\in\lbrack0,1]$,
\begin{align*}
\DRel(G,p) &  =\sum_{k=0}^{n}(-1)^{k}\left(  \sum_{l=0}^{k}(-1)^{l}%
d_{l}(G)\binom{n-l}{n-k}\right)  p^{k}\,,\\
\DRel(G,1-q) &  =\sum_{k=0}^{n}(-1)^{k}\left(  \sum_{l=0}^{k}(-1)^{l}%
d_{n-l}(G)\binom{n-l}{n-k}\right)  q^{k}\,.
\end{align*}
In particular, $\DRel(G,p)$ is a polynomial in $p$.
\end{corollary}

\begin{proof}
Both identities follow from Theorem \ref{thm1} by putting $q=1-p$,
respectively $p=1-q$, and applying the binomial theorem.
\end{proof}

As a consequence of Theorem \ref{thm1} in combination with Theorem \ref{thm2} 
we obtain a new inclusion-exclusion expansion 
and an analogue of Whitney's broken circuit theorem \cite{Whitney:1932}:

\begin{theorem}
\label{cor1} 
For any non-empty graph $G=(V,E)$ the domination polynomial satisfies
\begin{align}
\label{DominationIE}
D_{G}(x) \, = \, \sum_{J\subseteq V} (-1)^{|J|} (x+1)^{|V|-|N_{G}[J]|} \,\,\, = \!\! \sum_{J\subseteq V\atop |J|\le |V|-\delta(G)} \!\! (-1)^{|J|} \left[ (x+1)^{|V|-|N_{G}[J]|} - 1 \right] .
\end{align}
Moreover, 
if $G$ does not contain isolated vertices,
then for any linear ordering relation on $V$
and any system $\mathscr{X}$ of broken neighbourhoods of~$G$,
both sums in (\ref{DominationIE}) can be restricted to those subsets $J$ of the vertex-set
which do not include any $X\in\mathscr{X}$ as a subset.
\end{theorem}

\begin{proof}
For the first part of the theorem,
we may assume without loss of generality that $x>0$. 
Put $p=\frac{x}{x+1}$ and $q=\frac{1}{x+1}$.
By Theorem \ref{thm1} and the first part of Theorem \ref{thm2},
\[ D_G(x) = (x+1)^{|V|} \DRel\left(G;\frac{x}{x+1}\right) = (x+1)^{|V|} \sum_{J\subseteq V} (-1)^{|J|} (x+1)^{-|N_G[J]|} , \]
which proves the first identity. 
The second identity follows from the first one since, as in the proof of Corollary \ref{cor3},
$N_G[J]=V$ for any $J\subseteq V$, $|J|>|V|-\delta(G)$, and hence,
\[ \sum_{J\subseteq V} (-1)^{|J|} (x+1)^{|V|-|N_{G}[J]|} \,\,= \! \sum_{J\subseteq V\atop |J|\le |V|-\delta(G)} \!\! (-1)^{|J|} (x+1)^{|V|-|N_{G}[J]|} \,\,+ \! \sum_{J\subseteq V\atop |J|>|V|-\delta(G)}\!\! (-1)^{|J|} \, . \]
Since 
\begin{equation}
\label{evenodd}
 \sum_{J\subseteq V\atop |J|>|V|-\delta(G)} \! (-1)^{|J|}  \,\,= \,\,- \!\!\! \sum_{J\subseteq V\atop |J|\le|V|-\delta(G)} \!\! (-1)^{|J|} \, ,
\end{equation}
the second identity is proved. 
The second part of the theorem follows from the second part of Theorem~\ref{thm2}
and the following lemma, 
which implies that the sums in (\ref{evenodd}) can be restricted to the same $J\subseteq V$
which do not include any $X\in\mathscr{X}$ as a subset.
\end{proof}

The following lemma completes the proof of Theorem \ref{cor1}.

\begin{lemma}
Let $V$ be a non-empty linearly ordered set, 
and let $\mathscr{X}$ be a system of non-empty subsets of $V$ such that for any $X\in\mathscr{X}$, $\max X < \max V$.
Then there are as many even as odd subsets of $V$ which do not include any $X\in\mathscr{X}$ as a subset.
\end{lemma}

\begin{proof}
It is straightforward to check that $S\mapsto S\bigtriangleup \{\max V\}$
defines a one-to-one correspondence between the even and the odd subsets of $V$
not including any $X\in\mathscr{X}$.
\end{proof}

\begin{remark}
\label{reformulation}
For $V=\{1,\dots,n\}$, 
Theorem \ref{cor1} can equivalently be stated as 
\begin{multline}
\label{multlinelabel}
D_{G}(x) \,=\, \sum_{I\subseteq V}(-1)^{|V|-|I|}f_{G}(1_{I}(1),\dots,1_{I}(n);x) \\
         = \sum_{\genfrac{}{}{0pt}{}{I\subseteq V}{|I|\ge \delta(G)}}(-1)^{|V|-|I|}\left[f_{G}(1_{I}(1),\dots,1_{I}(n);x)-1\right]
\end{multline}
where $1_{I}$ denotes the indicator function of $I$, and $f_{G}$ the generating function
\[ f_{G}(i_{1},\dots,i_{n};x)\,:=\,\prod_{v\in V}\left(  1\,+\,x\!\prod_{w\in N_{G}[v]}\!i_{w}\right) . \]
Note that $I$ is a substitute for $V\setminus J$ in Theorem \ref{cor1}.
Therefore,
if $G$ does not contain isolated vertices,
then for any system $\mathscr{X}$ of broken neighbourhoods of~$G$,
both sums in (\ref{multlinelabel}) can be restricted to those subsets $I$ of $\{1,\dots,n\}$
that intersect any $X\in\mathscr{X}$.
\end{remark}

\begin{remark}
The contents of Remark \ref{remainsvalid1} analogously apply to Theorem \ref{cor1} and its preceding reformulation.
\end{remark}

\begin{corollary}
\label{newcor}
Under the requirements of Theorem \ref{cor1}, 
if $G$ neither contains isolated vertices nor isolated edges, then
\begin{equation*}
D_{G}(x) \, = \, \sum_{J\subseteq V\setminus A} (-1)^{|J|} (x+1)^{|V|-|N_{G}[J]|} \,\,\, = \!\! \sum_{J\subseteq V\setminus A\atop |J|\le |V|-\delta(G)} \!\! (-1)^{|J|} \left[ (x+1)^{|V|-|N_{G}[J]|} - 1 \right] .
\end{equation*}
where $A$ is any set of vertices of $G$ which are adjacent to a vertex of degree~1.
\end{corollary}

\begin{proof}
Corollary \ref{newcor} follows from Theorem \ref{cor1} in the same way as Corollary \ref{remainvalid2} follows from Theorem~\ref{thm2}.
\end{proof}

\section{Computational complexity}
\label{Complexity}

Apart from particular classes of graphs the methods developed in the previous sections
exhibit an exponential time behaviour.
There is not much hope to do better.

\begin{theorem} 
\label{complexity}
Computing the domination reliability polynomial of a graph is NP-hard.
\end{theorem}

\begin{proof}
We first note that computing the domination polynomial of a graph $G$ is NP-hard.
This follows immediately from a result of Flum and Grohe \cite{FlGr04}, 
who proved that computing $d_k(G)$ is $\#W[2]$-complete in the sense
of parametrized complexity.
Therefore, it remains to show that computing the domination polynomial of a graph is polynomially reducible 
to computing its domination reliability polynomial. \par
To this end, consider a graph $G$ having $n$ vertices, all operating randomly and independently with equal probability. 
Choose $n+1$ different values $p_{0},\dots,p_{n}\in\lbrack0,1]$, 
and put $q_{i}=1-p_{i}$ for $i=0,\dots,n$.
By Theorem~\ref{thm1},
\[ q_{i}^{-n}\,\DRel(G,p_{i}) \,=\, \sum_{k=0}^{n}d_{k}(G)\left(\frac{p_{i}}{q_{i}}\right)^{k} \quad (i=0,\dots,n). \]
This is a system of $n+1$ linear equations in the $n+1$ unknowns $d_{0}(G),\dots,d_{n}(G)$, 
which is uniquely solvable since its coefficient matrix is a Vandermonde matrix with non-zero determinant. 
The solution can be found in polynomial time by Gaussian elimination.
\end{proof}

\begin{corollary}
Computing the domination reliability of a graph $G$ is NP-hard.
\end{corollary}

\begin{corollary}
Computing the left to right (resp\@. right to left) domination reliability of a bipartite graph is NP-hard.
\end{corollary}

\section{Reliability of hypergraphs}
\label{Coverage}

In this section, we show that the domination reliability of a graph can be expressed 
in terms of the coverage probability of an associated hypergraph, and vice versa.

Recall that a \emph{hypergraph} is a couple $H=(V,\mathscr{E})$ 
where $V$ is a finite set and $\mathscr{E}$ is a set of non-empty subsets of $V$. 
The elements of $V$ resp. $\mathscr{E}$ are \emph{vertices} and \emph{edges} of $H$. 
A \emph{covering} of $V$ is a subset $\mathscr{X}$ of $\mathscr{E}$ such that $\bigcup\mathscr{X}=V$. 
In contrast to our discussion on graphs and their domination reliability, 
the vertices of the hypergraph are assumed to be perfectly reliable, 
whereas the edges fail randomly and independently with known probabilities, 
given by a vector $\mathbf{q}=\mathbf{1}-\mathbf{p}\in\lbrack0,1]^{\mathscr{E}}$. 
The \emph{coverage probability} of $H$, which is abbreviated to $\text{Cov}(H,\mathbf{p})$,
is the probability that the operating edges of $H$ constitute a covering of the vertex-set of $H$.

Given a hypergraph $H=(V,\mathscr{E})$ and a probability vector $\mathbf{p} = (p_{I})_{I\in\mathscr{E}}\in[0,1]^{\mathscr{E}}$, 
construct a graph $G$ on $V\cup\mathscr{E}$ with edge set
\[ \{ \{v,E\}\mathrel| v\in V,\, E\in\mathscr{E},\, v\in E\} \cup\{ \{E,F\}\mathrel| E,F\in\mathscr{E},\,E\neq F\} . \]
By defining $p_{v} = 0$ for any $v\in V$ 
(in addition to the given probabilities $p_{I}$ where $I\in\mathscr{E}$), 
a probability of operation is associated with any vertex of $G$. 
Evidently, by this construction, the coverage probability of $H$ equals the domination reliability of $G$.

On the other hand, the domination reliability of any graph can be expressed in
terms of the coverage probability of an associated hypergraph: 
Given a graph
$G=(V,E)$ and $\mathbf{p}=(p_{v})_{v\in V}\in\lbrack0,1]^{V}$, 
construct $H=(V,\mathscr{E})$ where $\mathscr{E}=\{N_{G}[v]\mathrel|v\in V\}$ 
and where for any $E\in\mathscr{E}$, 
$p_{E}$ is the probability that some vertex in $\{v\in V\mathrel|N_{G}[v]=E\}$ is operating in $G$. 
By the principle of inclusion-exclusion this probability can be computed as
\[
p_{E}=\sum_{\substack{I\subseteq V,\,I\neq\emptyset \\ N_G[i]=E \,\forall i\in I}}(-1)^{|I|-1}%
\prod_{i\in I}p_{i}\quad(E\in\mathscr{E}).
\]
By this construction, the domination reliability of $G$ equals the coverage probability of~$H$. 
Thus, the equivalence of the two concepts is shown.

The concept of coverage probability of hypergraphs goes back to \textsc{Ball}, \textsc{Provan}, and \textsc{Shier} \cite{BPS91}. 
The significance of this concept is due to the fact that 
the reliability of any coherent binary system can be expressed in terms of the coverage probability of an associated hypergraph, 
and vice versa (see \cite{BPS91,Shier91} for definition and details).

\section{Conclusion}
\label{Conclusions}

In this paper, we introduced a new network reliability measure called \emph{domination reliability}
and proved that its computation is NP-hard. 
Efficient solutions are only known for cographs.
For complete graphs, complete bipartite graphs, paths, and cycles explicit and recursive formul\ae{} can be given
in the case where all vertex reliabilities are equal. 
This case is of particular interest as it leads to a new graph polynomial,
which is called \emph{domination reliability polynomial},
and which is closely related to the domination polynomial---a graph polynomial which recently received considerable attention.
The authors propose to investigate the domination reliability polynomial in context with the domination polynomial, 
and to work out relationships with other graph polynomials. 
\par
From a practical point of view, the connection with the coverage probability of a hypergraph is significant:
Any coherent binary system can be transformed into a graph whose domination reliability equals the reliability of that system,
and vice versa.
Thus, the study of domination reliability may reveal some new insights into system reliability.


\end{document}